\documentclass[12pt,a4paper]{article}
\usepackage{latexsym, amssymb, amsthm}
\usepackage{dsfont}

\DeclareMathAlphabet\gothic{U}{euf}{m}{n}

\makeatletter
\def\eqnarray{\stepcounter{equation}\let\@currentlabel=\theequation
\global\@eqnswtrue
\tabskip\@centering\let\\=\@eqncr
$$\halign to \displaywidth\bgroup\hfil\global\@eqcnt\z@
  $\displaystyle\tabskip\z@{##}$&\global\@eqcnt\@ne
  \hfil$\displaystyle{{}##{}}$\hfil
  &\global\@eqcnt\tw@ $\displaystyle{##}$\hfil
  \tabskip\@centering&\llap{##}\tabskip\z@\cr}

\def\endeqnarray{\@@eqncr\egroup
      \global\advance\c@equation\m@ne$$\global\@ignoretrue}

\def\@yeqncr{\@ifnextchar [{\@xeqncr}{\@xeqncr[5pt]}}
\makeatother

\begin{document}
\bibliographystyle{tom}

\newtheorem{lemma}{Lemma}[section]
\newtheorem{thm}[lemma]{Theorem}
\newtheorem{cor}[lemma]{Corollary}
\newtheorem{prop}[lemma]{Proposition}

\theoremstyle{definition}

\newtheorem{remark}[lemma]{Remark}
\newtheorem{exam}[lemma]{Example}
\newtheorem{definition}[lemma]{Definition}

\newcommand{\gota}{\gothic{a}}
\newcommand{\gotb}{\gothic{b}}
\newcommand{\gotc}{\gothic{c}}
\newcommand{\gote}{\gothic{e}}
\newcommand{\gotf}{\gothic{f}}
\newcommand{\gotg}{\gothic{g}}
\newcommand{\gothh}{\gothic{h}}
\newcommand{\gotk}{\gothic{k}}
\newcommand{\gotl}{\gothic{l}}
\newcommand{\gotm}{\gothic{m}}
\newcommand{\gotn}{\gothic{n}}
\newcommand{\gotp}{\gothic{p}}
\newcommand{\gotq}{\gothic{q}}
\newcommand{\gotr}{\gothic{r}}
\newcommand{\gots}{\gothic{s}}
\newcommand{\gott}{\gothic{t}}
\newcommand{\gotu}{\gothic{u}}
\newcommand{\gotv}{\gothic{v}}
\newcommand{\gotw}{\gothic{w}}
\newcommand{\gotz}{\gothic{z}}
\newcommand{\gotA}{\gothic{A}}
\newcommand{\gotB}{\gothic{B}}
\newcommand{\gotG}{\gothic{G}}
\newcommand{\gotL}{\gothic{L}}
\newcommand{\gotS}{\gothic{S}}
\newcommand{\gotT}{\gothic{T}}

\newcounter{teller}
\renewcommand{\theteller}{(\alph{teller})}
\newenvironment{tabel}{\begin{list}%
{\rm  (\alph{teller})\hfill}{\usecounter{teller} \leftmargin=1.1cm
\labelwidth=1.1cm \labelsep=0cm \parsep=0cm}
                      }{\end{list}}

\newcounter{tellerr}
\renewcommand{\thetellerr}{(\roman{tellerr})}
\newenvironment{tabeleq}{\begin{list}%
{\rm  (\roman{tellerr})\hfill}{\usecounter{tellerr} \leftmargin=1.1cm
\labelwidth=1.1cm \labelsep=0cm \parsep=0cm}
                         }{\end{list}}

\newcounter{tellerrr}
\renewcommand{\thetellerrr}{(\Roman{tellerrr})}
\newenvironment{tabelR}{\begin{list}%
{\rm  (\Roman{tellerrr})\hfill}{\usecounter{tellerrr} \leftmargin=1.1cm
\labelwidth=1.1cm \labelsep=0cm \parsep=0cm}
                         }{\end{list}}

\newcounter{proofstep}
\newcommand{\nextstep}{\refstepcounter{proofstep}\vertspace \par 
          \noindent{\bf Step \theproofstep} \hspace{5pt}}
\newcommand{\firststep}{\setcounter{proofstep}{0}\nextstep}

\newcommand{\Ni}{\mathds{N}}
\newcommand{\Ki}{\mathds{K}}
\newcommand{\Qi}{\mathds{Q}}
\newcommand{\Ri}{\mathds{R}}
\newcommand{\Ci}{\mathds{C}}
\newcommand{\Ti}{\mathds{T}}
\newcommand{\Zi}{\mathds{Z}}
\newcommand{\Fi}{\mathds{F}}

\renewcommand{\proofname}{{\bf Proof}}

\newcommand{\simh}{{\stackrel{{\rm cap}}{\sim}}}
\newcommand{\ad}{{\mathop{\rm ad}}}
\newcommand{\Ad}{{\mathop{\rm Ad}}}
\newcommand{\alg}{{\mathop{\rm alg}}}
\newcommand{\clalg}{{\mathop{\overline{\rm alg}}}}
\newcommand{\Aut}{\mathop{\rm Aut}}
\newcommand{\arccot}{\mathop{\rm arccot}}
\newcommand{\capp}{{\mathop{\rm cap}}}
\newcommand{\rcapp}{{\mathop{\rm rcap}}}
\newcommand{\diam}{\mathop{\rm diam}}
\newcommand{\divv}{\mathop{\rm div}}
\newcommand{\dom}{\mathop{\rm dom}}
\newcommand{\codim}{\mathop{\rm codim}}
\newcommand{\RRe}{\mathop{\rm Re}}
\newcommand{\IIm}{\mathop{\rm Im}}
\newcommand{\tr}{{\mathop{\rm Tr \,}}}
\newcommand{\Tr}{{\mathop{\rm Tr \,}}}
\newcommand{\Vol}{{\mathop{\rm Vol}}}
\newcommand{\card}{{\mathop{\rm card}}}
\newcommand{\rank}{\mathop{\rm rank}}
\newcommand{\supp}{\mathop{\rm supp}}
\newcommand{\sgn}{\mathop{\rm sgn}}
\newcommand{\essinf}{\mathop{\rm ess\,inf}}
\newcommand{\esssup}{\mathop{\rm ess\,sup}}
\newcommand{\Int}{\mathop{\rm Int}}
\newcommand{\lcm}{\mathop{\rm lcm}}
\newcommand{\graph}{\mathop{\rm graph}}
\newcommand{\loc}{{\rm loc}}
\newcommand{\HS}{{\rm HS}}
\newcommand{\Trn}{{\rm Tr}}
\newcommand{\n}{{\rm N}}
\newcommand{\WOT}{{\rm WOT}}
\newcommand{\dd}{{\rm d}}
\newcommand{\rc}{{\rm c}}

\newcommand{\at}{@}

\newcommand{\mod}{\mathop{\rm mod}}
\newcommand{\spann}{\mathop{\rm span}}
\newcommand{\one}{\mathds{1}}

\hyphenation{groups}
\hyphenation{unitary}

\newcommand{\tfrac}[2]{{\textstyle \frac{#1}{#2}}}

\newcommand{\ca}{{\cal A}}
\newcommand{\cb}{{\cal B}}
\newcommand{\cc}{{\cal C}}
\newcommand{\cd}{{\cal D}}
\newcommand{\ce}{{\cal E}}
\newcommand{\cf}{{\cal F}}
\newcommand{\ch}{{\cal H}}
\newcommand{\chs}{{\cal HS}}
\newcommand{\ci}{{\cal I}}
\newcommand{\ck}{{\cal K}}
\newcommand{\cl}{{\cal L}}
\newcommand{\cm}{{\cal M}}
\newcommand{\cn}{{\cal N}}
\newcommand{\co}{{\cal O}}
\newcommand{\cp}{{\cal P}}
\newcommand{\cs}{{\cal S}}
\newcommand{\ct}{{\cal T}}
\newcommand{\cv}{{\cal V}}
\newcommand{\cx}{{\cal X}}
\newcommand{\cy}{{\cal Y}}
\newcommand{\cz}{{\cal Z}}

\thispagestyle{empty}

\vspace*{1cm}
\begin{center}
{\Large\bf Nittka's invariance criterion and \\[5pt]
Hilbert space 
valued parabolic equations in $L_p$} \\[4mm]

\large W. Arendt$^1$, A.F.M. ter Elst$^2$ and M. Sauter$^1$

\end{center}

\vspace{4mm}

\begin{center}
{\bf Abstract}
\end{center}

\begin{list}{}{\leftmargin=1.8cm \rightmargin=1.8cm \listparindent=10mm 
   \parsep=0pt}
\item
Nittka gave an efficient criterion on a form defined on $L_2(\Omega)$ 
which implies that the associated semigroup is $L_p$-invariant
for some given $p \in (1,\infty)$.
We extend this criterion to the Hilbert space valued~$L_2(\Omega,H)$.
As an application we consider elliptic systems of pure second order.
Our main result shows that the induced semigroup is $L_p$-contractive
for all $p \in [p_-,p_+]$ for some $1 < p_- < 2 < p_+ < \infty$.
\end{list}

\vspace{6mm}
\noindent
October 2023.

\vspace{3mm}

\vspace{3mm}

\vspace{6mm}

\noindent
{\bf Home institutions:}    \\[3mm]
\begin{tabular}{@{}cl@{\hspace{10mm}}cl}
1. & Institute of Applied Analysis & 
2. & Department of Mathematics  \\
& Ulm University & 
  & University of Auckland  \\
& Helmholtzstr.\ 18  &
  & Private bag 92019 \\
& 89081 Ulm  &
   & Auckland 1142 \\ 
& Germany  &
  & New Zealand \\[8mm]
\end{tabular}

\newpage

\section{Introduction} \label{Spacc1}

Let $(S_t)_{t > 0}$ be a $C_0$-semigroup on $L_2(\Omega)$, 
with $\Omega \subset \Ri^d$, which is associated with a closed 
sesquilinear form.
Ouhabaz \cite{Ouh} gave a convenient criterion on the 
form characterising when the semigroup is $L_\infty$-contractive.
It is more complicated to describe $L_p$-contractivity if 
$1 < p < \infty$.
The reason is the fact that there is no explicit formula which describes
the orthogonal projection from $L_2(\Omega)$ to the 
closed convex set $ \{ u \in L_2(\Omega) : \|u\|_p \leq 1 \} $
if $1 < p < \infty$ and $p \neq 2$.

Nonetheless, Nittka \cite{Nit5} succeeded to overcome the difficulty by a structural 
analysis and developed an efficient criterion for the 
$L_p$-contractivity of a $C_0$-semigroup on $L_2(\Omega)$ 
which is associated with a form.

The purpose of this paper is twofold.
Our first aim is to present Nittka's result.
We do this within the more general setting of the 
vector-valued space $L_2(\Omega,H)$, where $H$ is a Hilbert space.
The result is the following.

\begin{thm} \label{tpacc101}
Let $(\Omega,\cb,\mu)$ be a $\sigma$-finite measure space.
Let $H$ be a Hilbert space.
Fix $p \in (1,\infty)$.
Define 
\[
C = \{ u \in L_2(\Omega,H) \cap L_p(\Omega,H) : \|u\|_{L_p(\Omega,H)} \leq 1 \} 
.  \]
Let $P$ be the orthogonal projection of $L_2(\Omega,H)$ onto $C$.
Let $\cv$ be a Hilbert space which is continuously and 
densely embedded in $L_2(\Omega,H)$.
Let $\gota \colon \cv \times \cv \to \Ci$ be a continuous elliptic sesquilinear form,
let $A$ be the operator in $L_2(\Omega,H)$ associated with $\gota$
and let $S$ be the semigroup generated by $-A$.
Then the following are equivalent.
\begin{tabeleq}
\item \label{tpacc101-1}
$\|S_t u\|_{L_p(\Omega,H)} \leq \|u\|_{L_p(\Omega,H)}$
for all $u \in L_2(\Omega,H) \cap L_p(\Omega,H)$ and $t > 0$.
\item \label{tpacc101-2}
$P \cv \subset \cv$ and
$\RRe \gota(u, \|u\|_H^{p-1} \, \sgn u) \geq 0$
for all $u \in \cv$
with $\|u\|_H^{p-1} \, \sgn u \in \cv$.
\end{tabeleq}
\end{thm}

Our proof in Section~\ref{Spacc2} is slightly different from Nittka's 
since we exploit 
the strict convexity of the space $L_p(\Omega,H)$ for 
all $1 < p < \infty$, which we prove in Appendix~\ref{Spaccapp2}.

Our second aim is to apply the criterion to 
purely second-order Hilbert space valued 
elliptic operators with Neumann boundary conditions. 
They generate a contractive $C_0$-semigroup $(S_t)_{t > 0}$ 
on $L_2(\Omega,H)$.
Of particular interest are systems, that is $H = \Ci^d$.
In Section~\ref{Spacc3} we show that there is an interval $[p_-,p_+]$, with 
$1 < p_- < 2 < p_+ < \infty$, such that the semigroup $S$ 
extends to a contractive $C_0$-semigroup on $L_p(\Omega,H)$
for all $p \in [p_-,p_+]$.
To prove the needed estimates, we use a chain rule
formula, which is quite delicate and will be proved in Appendix~\ref{Spaccapp1}.
In the scalar case our results may be compared with 
Cialdea--Maz'ya \cite{CiaM},
who introduced an algebraic version of $L_p$-dissipativity
and presented an algebraic characterisation for scalar-valued 
elliptic operators.
This algebraic characterisation was refined by 
Carbonaro--Dragi{\v{c}}evi{\'c}
\cite{CarbonaroDragicevic1} and they used the 
result of Nittka to describe contractive $C_0$-semigroups on $L_p(\Omega)$
via the notion that they called $p$-ellipticity.

\section{Nittka's criterion for $L_p$-contractivity} \label{Spacc2}
Let $(\Omega,\cb,\mu)$ be a $\sigma$-finite measure space.
Let $H$ be a Hilbert space.
For all $p \in [1,\infty)$ we write $L_p = L_p(\Omega,H)$.
If $u \in L_p$, then we write $\|u\|_p = \|u\|_{L_p(\Omega,H)}$
and $\|u\|_H \colon \Omega \to \Ri$ is the function from $\Omega$ into $\Ri$ 
such that 
\[
\|u\|_H(x) = \|u(x)\|_H
\]
for all $x \in \Omega$.
Further we write $\ch = L_2 = L_2(\Omega,H)$.
Throughout this paper we fix $p \in (1,\infty)$.
Define 
\[
C = \{ u \in \ch \cap L_p : \|u\|_p \leq 1 \} 
.  \]
Clearly $C$ is convex and it follows from Fatou's lemma that $C$ is 
closed in~$\ch$.
Let $P \colon \ch \to C$ be the orthogonal projection.
For all $u \in \ch$ define 
\[
N(u) 
= \{ h \in \ch : \RRe (h, v-u)_\ch \leq 0 \mbox{ for all } v \in C \} 
.  \]
Then $N(u)$ is a closed cone in~$\ch$ with $0 \in N(u)$.
We state an easy property regarding $C$ and $N(u)$.

\begin{lemma} \label{lpacc201}
For all $f \in \ch$ there exist unique $u \in C$ and $h \in N(u)$
such that $f = u + h$.
Actually, $u = Pf$.
\end{lemma}
\begin{proof}
Let $u = Pf$ and $h = f - Pf$.
Then $u \in C$ and 
$\RRe (h, v-u)_\ch = \RRe (f - Pf, \penalty-10 v - Pf)_\ch \leq 0$
for all $v \in C$.
This proves existence.
Let also $\tilde u \in C$ and $\tilde h \in N(\tilde u)$ be such that 
$f = \tilde u + \tilde h$.
Then $\RRe(h, \tilde u - u)_\ch \leq 0$ and similarly
$\RRe (\tilde h, u - \tilde u)_\ch \leq 0$.
Hence $\RRe (h - \tilde h, \tilde u - u)_\ch \leq 0$.
Since $h - \tilde h = \tilde u - u$ this implies that 
$\RRe \|h - \tilde h\|_\ch^2 \leq 0$ and the statement follows.
\end{proof}

If $u \colon \Omega \to H$ is a function, then define
$\sgn u \colon \Omega \to H$ by 
\[
(\sgn u)(x)
= \left\{ \begin{array}{ll}
   \frac{1}{\|u(x)\|_H} \, u(x) & \mbox{if } u(x) \neq 0,  \\[5pt]
   0 & \mbox{if } u(x) = 0 .
          \end{array} \right.
\]
Note that $\sgn u = \lim_{\varepsilon \downarrow 0} u_\varepsilon$
pointwise, where the function $u_\varepsilon \colon \Omega \to H$ is defined by 
$u_\varepsilon(x) = \frac{1}{\sqrt{\|u(x)\|_H^2 + \varepsilon}} \, u(x)$.
Hence $\sgn u$ is (Bochner) measurable if $u$ is (Bochner) measurable.
Then also $\|u\|_H^{p-1} \, \sgn u$ is (Bochner) measurable 
whenever $u$ is (Bochner) measurable.

For a description of $N(u)$ we use that $L_p(\Omega,H)$ is
strictly convex.
See Appendix~\ref{Spaccapp2} for a proof.
Recall that a Banach space $E$ is called {\bf strictly convex}
if for all $\xi,\eta \in E$ with $\|\xi\|_E = 1 = \|\eta\|_E$ and $\xi \neq \eta$
it follows that $\|\xi + \eta\|_E < 2$.

\begin{lemma} \label{lpacc203}
Let $E$ be a Banach space.
Assume that $E^*$ strictly convex
and let $x \in E$.
Then there exists a unique $f \in E^*$ such that 
$\RRe f(x) = \|x\|_E^2 = \|f\|_{E^*}^2$.
This unique $f$ satisfies $f(x) = \|x\|_E^2$.
\end{lemma}
\begin{proof}
The existence of an $f \in E^*$ such that $f(x) = \|x\|_E^2 = \|f\|_{E^*}^2$
is a well-known consequence of the Hahn--Banach theorem.
For the uniqueness, we may assume without loss of generality that 
$\|x\|_E = 1$.
Suppose also $g \in E^*$ with $\RRe g(x) = \|x\|_E^2 = \|g\|_{E^*}^2$
and $g \neq f$.
Define $h = \frac{1}{2} \, (f + g)$.
Then $\|h\|_{E^*} < 1$ by the strict convexity of $E^*$.
But then 
\[
1 
= \tfrac{1}{2} \, \RRe (f(x) + g(x))
= \RRe h(x)
\leq \|h\|_{E^*} \, \|x\|_E
< 1
.  \]
This is a contradiction.
\end{proof}

Now we are able to give a characterisation for $N(u)$.

\begin{prop} \label{ppacc204}
Let $u \in C$.
Then the following are equivalent.
\begin{tabeleq}
\item \label{ppacc204-1}
$N(u) \neq \{ 0 \} $.
\item \label{ppacc204-2}
$\|u\|_p = 1$ and $\|u\|_H^{p-1} \, \sgn u \in \ch$.
\end{tabeleq}
If these conditions are valid, then 
$N(u) = \{ t \, \|u\|_H^{p-1} \, \sgn u : t \in [0,\infty) \} $.
\end{prop}
\begin{proof}
`\ref{ppacc204-2}$\Rightarrow$\ref{ppacc204-1}'.
Let $v \in C$.
Then the Cauchy--Schwarz inequality and the H\"older inequality give
\begin{eqnarray*}
\RRe (\|u\|_H^{p-1} \, \sgn u, v)_\ch
& \leq & \int_\Omega \|u\|_H^{p-1} \, \|v\|_H  \\
& \leq & \Big( \int_\Omega \|u\|_H^{(p-1) p'} \Big)^{1/p'} \, \|v\|_p  \\
& \leq & \|u\|_p^{p/p'}
= 1
= \RRe (\|u\|_H^{p-1} \, \sgn u, u)_\ch
.
\end{eqnarray*}
So $\|u\|_H^{p-1} \, \sgn u \in N(u)$
and $ \{ t \, \|u\|_H^{p-1} \, \sgn u : t \in [0,\infty) \} \subset N(u)$.
In particular, $N(u) \neq \{ 0 \} $.

`\ref{ppacc204-1}$\Rightarrow$\ref{ppacc204-2}'.
We first show that $\|u\|_p = 1$.
Suppose that $\|u\|_p < 1$.
Let $h \in N(u)$.
Let $w \in L_2 \cap L_p$ with $\|w\|_p \leq 1 - \|u\|_p$.
Then $u + w \in C$.
Since $h \in N(u)$ one deduces that $\RRe (h,w)_\ch \leq 0$.
Because $L_2 \cap L_p$ is dense in $\ch$ it follows that $h = 0$.
Hence \ref{ppacc204-2} implies that $\|u\|_p = 1$.

Next let $h \in N(u)$ with $h \neq 0$.
If $v \in L_2(\Omega,H) \cap L_p(\Omega,H)$, then 
$\RRe (h,v)_\ch \leq \|v\|_p \, \RRe (h,u)_\ch$.
So $h \in L_{p'}(\Omega,H)$ and $\|h\|_{p'} \leq \RRe (h,u)_\ch$.
If $\RRe (h,u)_\ch = 0$, then $\|h\|_{p'} = 0$ and $h = 0$, which is 
a contradiction.
So $\RRe (h,u)_\ch \neq 0$.
Multiplying $h$ with a strictly positive constant we may assume 
that $\RRe (h,u)_\ch = 1$.
Then 
\[
\|h\|_{p'}
\leq 1
= \RRe (h,u)_\ch
\leq \|h\|_{p'} \, \|u\|_p
\leq \|h\|_{p'}
.  \]
So $\|h\|_{p'} = 1 = \RRe (h,u)_\ch = \RRe \langle h,u \rangle_{L_{p'} \times L_p}$.
We proved that 
\[
\RRe \langle h,u \rangle_{L_{p'} \times L_p}
= \|h\|_{p'}^2 
= \|u\|_p^2
.  \]
On the other hand, 
\[
\int_\Omega \Big\| \|u\|_H^{p-1} \, \sgn u \Big\|_H^{p'}
= \int_\Omega \|u\|_H^{(p-1) p'}
= \int_\Omega \|u\|_H^p
= 1
, \]
so $\|u\|_H^{p-1} \, \sgn u \in L_{p'}$ and furthermore
\[
\RRe \langle \|u\|_H^{p-1} \, \sgn u,u \rangle_{L_{p'} \times L_p}
= \int_\Omega \|u\|_H^p
= 1
= \Big\| \|u\|_H^{p-1} \, \sgn u \Big\|_{p'}^2 
= \|u\|_p^2
.  \]
By Lemma~\ref{lpacc203} we obtain that 
$\|u\|_H^{p-1} \, \sgn u = h \in \ch$.
Moreover, $N(u) \subset \{ t \, \|u\|_H^{p-1} \, \sgn u : t \in [0,\infty) \} $.
\end{proof}

Let $\cv$ be a Hilbert space that is continuously and densely embedded in 
$\ch = L_2(\Omega,H)$.
Let $\gota \colon \cv \times \cv \to \Ci$ be a sesquilinear form such 
that $\gota$ is {\bf continuous}, that is there is an $M > 0$ such that 
\[
|\gota(u,v)|
\leq M \, \|u\|_\cv \, \|v\|_\cv
\]
for all $u,v \in \cv$, and $\gota$ is {\bf elliptic}, that is 
there are $\mu > 0$ and $\omega \in \Ri$ such that 
\begin{equation}
\RRe \gota(u,u) + \omega \, ||u\|_\ch^2
\geq \mu \, \|u\|_\cv^2
\label{eSpacc2;3}
\end{equation}
for all $u \in \cv$.
Then there is a unique operator $A$ in $\ch$ whose graph is 
\[
\graph(A)
= \{ (u,f) : u \in \cv, \; f \in \ch \mbox{ and }
   \gota(u,v) = (f,v)_\ch \mbox{ for all } v \in \cv \} 
.  \]
We call $A$ the {\bf operator associated with the form $\gota$}.
Then $-A$ generates a holomorphic $C_0$-semigroup $(S_t)_{t > 0}$
in $\ch$ satisfying $\|S_t\|_{\ch \to \ch} \leq e^{\omega t}$
for all $t > 0$, where $\omega$ is as in (\ref{eSpacc2;3}).

\begin{proof}[{\bf Proof of Theorem~\ref{tpacc101}.}]
`\ref{tpacc101-1}$\Rightarrow$\ref{tpacc101-2}'.
It follows from \cite{Ouh5} Theorem~2.2 1)$\Rightarrow$2) that 
$P \cv \subset \cv$ and $\RRe \gota(Pf, f - Pf) \geq 0$ for all $f \in \cv$.
Since $\|u\|_H^{p-1} \, \sgn u \in L_2(\Omega,H)$ by assumption,
one deduces that 
\[
\int_\Omega \|u\|_H^p 
= \int_\Omega ( \|u\|_H^{p-1} \, \sgn u, \|u\|_H \, \sgn u)_H
< \infty
.  \]
So $u \in L_p$.
Without loss of generality we may assume that $\|u\|_p = 1$.
Then $u \in C$ and $\|u\|_H^{p-1} \, \sgn u \in \ch$, so 
$\|u\|_H^{p-1} \, \sgn u \in N(u)$ by Proposition~\ref{ppacc204}.
Set $f = u + \|u\|_H^{p-1} \, \sgn u \in \cv \subset \ch$.
Then the uniqueness of Lemma~\ref{lpacc201} gives $u = Pf$.
Consequently
$\RRe \gota(u, \|u\|_H^{p-1} \, \sgn u) = \RRe \gota(Pf, f - Pf) \geq 0$.

`\ref{tpacc101-2}$\Rightarrow$\ref{tpacc101-1}'.
Let $v \in \cv$.
We shall show that $\RRe \gota(Pv, v - Pv) \geq 0$.
If $v \in C$, then this is trivial, so we may assume that $v \not\in C$.
Set $u = Pv$ and $h = v - u$.
Then $h \in N(u)$ by Lemma~\ref{lpacc201}.
Also $h \neq 0$, so Proposition~\ref{ppacc204} implies that 
$\|u\|_p = 1$ and 
$\|u\|_H^{p-1} \, \sgn u \in \ch$.
Moreover, there exists a 
$t \in [0,\infty)$ such that $h = t \, \|u\|_H^{p-1} \, \sgn u$.
Then $t \neq 0$.
Since $u = Pv \in P \cv \subset \cv$ by assumption,
one deduces that $\|u\|_H^{p-1} \, \sgn u = \frac{1}{t} (v-u) \in \cv$.
Hence 
$\RRe \gota(Pv, v - Pv) = t \RRe \gota(u, \|u\|_H^{p-1} \, \sgn u) \geq 0$.
Now it follows from \cite{Ouh5} Theorem~2.2 2)$\Rightarrow$1) that 
$S_t C \subset C$ for all $t > 0$.
This obviously implies Statement~\ref{tpacc101-1}.
\end{proof}

\section{Application to Hilbert space valued parabolic problems} \label{Spacc3}

Let $\Omega \subset \Ri^d$ be an open set.
Let $H$ be a separable Hilbert space.
For all $k,l \in \{ 1,\ldots,d \} $ let 
$c_{kl} \colon \Omega \to \cl(H)$ be a bounded function
such that $x \mapsto (c_{kl}(x) \, \xi, \eta)_H$ is 
measurable from $\Omega$ into $\Ci$ for all $\xi,\eta \in H$.
Let $\mu > 0$.
We assume that 
\[
\RRe \sum_{k,l=1}^d (c_{kl}(x) \, \xi_l, \xi_k)_H
\geq \mu \sum_{k=1}^d \|\xi_k\|_H^2
\]
for all $x \in \Omega$ and $\xi_1,\ldots,\xi_d \in H$.
Further let $M > 0$ be such that 
\[
\sum_{k=1}^d \Big\| \sum_{l=1}^d c_{kl}(x) \, \xi_l \Big\|_H^2
\leq M^2 \, \sum_{k=1}^d \|\xi_k\|_H^2
\]
for all $x \in \Omega$ and $\xi_1,\ldots,\xi_d \in H$.
For simplicity we consider Neumann boundary conditions.
Define $\cv = H^1(\Omega,H)$ and $\gota \colon \cv \times \cv \to \Ci$ by
\[
\gota(u,v)
= \sum_{k,l=1}^d \int_\Omega (c_{kl}(x) \, (\partial_l u)(x), (\partial_k v)(x) )_H \, dx
.  \]
Then $\gota$ is a continuous elliptic sesquilinear form.
Let $A$ be the operator associated with $\gota$ and let $S$ be the semigroup
generated by $-A$.

\begin{thm} \label{tpacc301}
Let $p \in (1,\infty)$ and suppose that 
\[
\frac{\mu}{M}
\geq 2 \Big| \frac{p-2}{p} \Big| + \Big| \frac{p-2}{p} \Big|^2
.  \]
Then $S$ extends consistently to a contraction semigroup in $L_p(\Omega,H)$.
\end{thm}

Note that the condition in Theorem~\ref{tpacc301} is invariant 
by taking the dual exponent, that is, if $p \in (1,\infty)$ 
satisfies the condition, then so does $q$, where $\frac{1}{p} + \frac{1}{q} = 1$.
If one is satisfied with some small interval, $1 < p_- < 2 < p_+ < \infty$
such that $S$ is $L_p$-contractive for all $p \in [p_-,p^+]$, 
then a significantly easier and less technical proof than the following 
can be given.

\begin{proof}
Using duality, without loss of generality we may assume that $p > 2$.
We argue as in \cite{CiaM}, \cite{ELSV} and \cite{Egert2}.
Let $u \in H^1(\Omega,H)$.
For all $n \in \Ni$ define 
\[
v_n = (\|u\|_H^{\frac{p-2}{2}} \wedge n) \, u
\quad , \quad
w_n = (\|u\|_H^{p-2} \wedge n^2) \, u
\quad \mbox{and} \quad
\chi_n = \one_{[\|u\|_H^{p-2} < n^2]}
.  \]
If follows from Proposition~\ref{ppaccapp101} (cf.\ \cite{ArendtKreuter} Theorems~3.3 and 4.2)
that $v_n,w_n \in H^1(\Omega,H)$
with 
\begin{eqnarray}
\nabla v_n 
& = & n \, (\one - \chi_n) \, \nabla u 
    + \chi_n \, \|u\|_H^{\frac{p-2}{2}} 
         \Big( \nabla u + \tfrac{p-2}{2} \, (\nabla \|u\|_H) \, \sgn u \Big) , \mbox{ and} \nonumber  \\
\nabla w_n 
& = & n^2 \, (\one - \chi_n) \, \nabla u 
    + \chi_n \, \|u\|_H^{p-2}
         \Big( \nabla u + (p-2) \, (\nabla \|u\|_H) \, \sgn u \Big)  \label{etpacc301;1}
.
\end{eqnarray}
Note that $\chi_n \, \|v_n\|_H = \chi_n \, \|u\|_H^{p/2}$.
Hence 
\begin{eqnarray*}
(\one - \chi_n) \, \nabla u 
& = & \frac{1}{n} \, (\one - \chi_n) \, \nabla v_n ,  \\
\chi_n \, \|v_n\|_H^{\frac{p-2}{p}} \, \nabla u
& = & \chi_n \, \Big( \nabla v_n - \tfrac{p-2}{p} \, (\nabla \|v_n\|_H) \, \sgn v_n \Big), \mbox{ and}  \\
\nabla w_n
& = & n \, (\one - \chi_n) \, \nabla v_n
   + \chi_n \, \, \|v_n\|_H^{\frac{p-2}{p}} 
         \Big( \nabla v_n + \tfrac{p-2}{p} \, (\nabla \|v_n\|_H) \, \sgn v_n \Big) .
\end{eqnarray*}
Therefore for almost all $x \in \Omega$ one obtains
\begin{eqnarray*}
\lefteqn{
\sum_{k,l=1}^d \RRe (c_{kl} \, \partial_l u, \partial_k w_n )_H
} \hspace*{3mm} \\*
& = & (\one - \chi_n) \sum_{k,l=1}^d \RRe (c_{kl} \, \partial_l u, n \, \partial_k v_n )_H
\\*
& & \hspace*{10mm} {}
   + \chi_n \sum_{k,l=1}^d \RRe (c_{kl} \, \|v_n\|_H^{\frac{p-2}{p}} \, \partial_l u, 
         \Big( \partial_k  v_n + \tfrac{p-2}{p} \, (\partial_k  \|v_n\|_H) \, \sgn v_n \Big) )_H  \\
& = & (\one - \chi_n) \sum_{k,l=1}^d \RRe (c_{kl} \, \partial_l v_n, \partial_k v_n )_H
\\*
& & \hspace*{1mm} {}
   + \chi_n \sum_{k,l=1}^d \RRe (c_{kl} \, \Big( \partial_l v_n - \tfrac{p-2}{p} \, (\partial_l \|v_n\|_H) \, \sgn v_n \Big) , 
         \Big( \partial_k  v_n + \tfrac{p-2}{p} \, (\partial_k  \|v_n\|_H) \, \sgn v_n \Big) )_H   \\
& = & \sum_{k,l=1}^d \RRe (c_{kl} \, \partial_l v_n, \partial_k v_n )_H
   + \tfrac{p-2}{p} \, \chi_n \, \RRe ( (c_{kl} - c_{lk}^*) \, \partial_l v_n, 
           (\partial_k  \|v_n\|_H) \, \sgn v_n )_H 
\\*
& & \hspace*{10mm} {}
   - \Big( \tfrac{p-2}{p} \Big)^2 \, 
        \RRe (c_{kl} \, (\partial_l \|v_n\|_H) \, \sgn v_n, (\partial_k \|v_n\|_H) \, \sgn v_n )_H  \\
& \geq & \Big( \mu - 2 M \, \tfrac{p-2}{p} - M \, \Big( \tfrac{p-2}{p} \Big)^2 \Big) 
             \sum_{k=1}^d \|\partial_k v_n\|_H^2  
,
\end{eqnarray*}
where we used that 
$\sum_{k=1}^d | \partial_k \|v_n\|_H |^2 \leq \sum_{k=1}^d \|\partial_k v_n\|_H^2$
and $\|\sgn v_n\|_H \leq 1$ by Lemma~\ref{lpaccapp103.7}.
By the assumption on $p$ we obtain that 
\[
\sum_{k,l=1}^d \RRe (c_{kl} \, \partial_l u, \partial_k w_n )_H
\geq 0
\]
almost everywhere.
This is for all $n \in \Ni$.

Next take the limit $n \to \infty$ and use (\ref{etpacc301;1}).
Then 
\begin{equation}
\sum_{k,l=1}^d \RRe (c_{kl} \, \partial_l u, 
                   \|u\|_H^{p-2} \, (\partial_k u + (p-2) \, (\partial_k \|u\|_H) \, \sgn u ))_H
\geq 0
\label{etpacc301;4}
\end{equation}
almost everywhere.

We assume from now on that in addition $\|u\|_H^{p-2} \, u \in \cv$.
Let $k \in \{ 1,\ldots,d \} $.
We shall show that $\partial_k (\|u\|_H^{p-2} \, u) = f_k$ almost everywhere, 
where 
\[
f_k = \|u\|_H^{p-2} \, \Big( \partial_k u + (p-2) \, (\partial_k \|u\|_H) \, \sgn u \Big)
.  \]
Write $r = 2 \, \frac{p-1}{p-2} \in (2,\infty)$ and let $q \in (1,2)$ be such that 
$\frac{1}{2} + \frac{1}{r} = \frac{1}{q}$.
Since $\|u\|_H^{p-1} \in L_2(\Omega)$ one deduces that 
$\int_\Omega (\|u\|_H^{p-2})^r = \int_\Omega (\|u\|_H^{p-1})^2 < \infty$.
So $\|u\|_H^{p-2} \in L_r(\Omega)$.
If $n \in \Ni$, then 
$n^2 \, (\one - \chi_n) \, \|\partial_k u\|_H 
\leq (\one - \chi_n) \, \|u\|_H^{p-2} \, \|\partial_k u\|_H$
and hence (\ref{etpacc301;1}) gives
\begin{equation}
\|\partial_k w_n\|_H 
\leq \|u\|_H^{p-2} \Big( \|\partial_k u\|_H + (p-2) \, \Big| \partial_k \|u\|_H \Big| \Big)
\label{etpacc301;2}
\end{equation}
Note that the right hand side of (\ref{etpacc301;2})
does not depend on $n$ and is an element of $L_q(\Omega)$.
Also $\lim \partial_k w_n = f_k$ almost everywhere.
Hence $\lim \partial_k w_n = f_k$ in $L_q(\Omega,H)$.
It is easy to see that $\lim w_n = \|u\|_H^{p-2} \, u$ in $L_2(\Omega,H)$
since $\|u\|_H^{p-2} \, u \in L_2(\Omega,H)$.
Let $\varphi \in C_c^\infty(\Omega,H)$.
Then 
\begin{eqnarray*}
\int_\Omega (\partial_k (\|u\|_H^{p-2} \, u), \varphi)_H
& = & - \int_\Omega (\|u\|_H^{p-2} \, u, \partial_k \varphi)_H  \\
& = & - \lim_{n \to \infty} \int_\Omega (w_n, \partial_k \varphi)_H
= \lim_{n \to \infty} \int_\Omega (\partial_k w_n, \varphi)_H 
= \int_\Omega (f_k, \varphi)_H
.  
\end{eqnarray*}
So $\partial_k (\|u\|_H^{p-2} \, u) = f_k$ almost everywhere.

Finally (\ref{etpacc301;4}) implies that 
\[
\sum_{k,l=1}^d \RRe (c_{kl} \, \partial_l u,  \partial_k (\|u\|_H^{p-2} \, u))_H
\geq 0
\]
almost everywhere.
Integrating over $\Omega$ gives $\RRe \gota(u,\|u\|_H^{p-2} \, u) \geq 0$.
Now apply Theorem~\ref{tpacc101}.

By the way, with some more work one can show that $\lim_{n \to \infty} \partial_k w_n = f_k$
in $L_2(\Omega,H)$.
\end{proof}

We comment on related results concerning the extension of $S$ to a not 
necessarily contractive $C_0$-semigroup on $L_p$.
Hofmann, Mayboroda and McIntosh~\cite{HMM} showed for $H=\Ci$, $\Omega=\Ri^d$ 
and $d\ge 3$ that the semigroup $S$ can be extended to a $C_0$-semigroup 
on $L_p(\Omega)$ if $p\in[\frac{2d}{d+2},\frac{2d}{d-2}]$. 
Conversely, for each $p\in(1,\frac{2d}{d+2})\cup (\frac{2d}{d-2},\infty)$
they construct an elliptic operator such that the associated semigroup $(S_t)_{t>0}$ 
cannot be extended consistently to a bounded semigroup $L_p(\Ri^d)$.
The extension results are based on off-diagonal Davies--Gaffney estimates;
 cf.~also \cite{Dav10} Theorem~25 and~\cite{Aus2} Section~3.1.

Davies had already pointed out in the introduction of \cite{Dav11} 
that his proof of~\cite{Dav10} Theorem~25 extends to the vector-valued case.
Moreover, by~\cite{Dav11} Theorem~10 for each 
$p\in (1,\frac{2d}{d+2})\cup(\frac{2d}{d-2},\infty)$ there exists an 
elliptic system with $H=\Ci^d$, $\Omega=\Ri^d$, $d\ge 3$ and with
real symmetric coefficients such that the operator $S_t$ 
does not continuously extend to $L_p$ for any $t>0$.

We shall give a corresponding extension result for our setting,
which we obtain readily from standard estimates and tracing 
Auscher's proof of~\cite{Aus2} Proposition~3.2.

\begin{thm} \label{thm:vv-auscher}
Suppose $\Omega=\Ri^d$ or $\Omega\subset\Ri^d$ is open and Lipschitz. 
Then the semigroup $S$ extends to a $C_0$-semigroup with growth bound $0$ 
on $L_p(\Omega,H)$ for all $p\in(\frac{2d}{d+2},\frac{2d}{d-2})$ 
if $d\ge 3$ and for all $p\in(1,\infty)$ if $d\in\{1,2\}$.
\end{thm}
\begin{proof}
We outline the arguments for $d\ge 3$.
Let $\omega>0$.
Since $\gota$ is elliptic for this choice of $\omega$,
by~\cite{Tan} Lemma~3.6.2\,(3.60) there exists a $c>0$ such that 
$\|e^{-\omega t} \, S_t u\|_\cv \le c \, t^{-1/2} \, \|u\|_2$ for all $u\in L_2(\Omega,H)$
and $t > 0$.
Combining this with the Sobolev embedding $\cv \hookrightarrow L_{2d/(d-2)}$, 
we obtain that there exists a $C>0$ such that
\[
\|e^{-\omega t} \, S_t u\|_{2d/(d-2)} \le C \, t^{-1/2} \, \|u\|_2
\]
for all $u\in L_2(\Omega,H)$ and $t > 0$.
Then by duality
\[
\|e^{-\omega t} \, S_t^* u\|_2 \le C \, t^{-1/2} \, \|u\|_{2d/(d+2)}
\]
for all $u \in L_2 \cap L_{2d/(d+2)}$ and $t > 0$.

Next, it follows from inspection of the proof of~\cite{Aus2} Proposition~3.2
that the parts (2) and (3) of \cite{Aus2} Proposition~3.2
extend to the vector-valued case and 
general open sets $\Omega$, and are applicable to the $C_0$-semigroup 
$(T_t)_{t>0}$ given by $T_t = e^{-\omega t} \, S_t^*$.
For the extension of part~(2) one needs $L_2$--$L_2$ off-diagonal 
estimates that can be obtained, for example, as in \cite{AE2} Theorem~4.2.
Moreover, the vector-valued version of the Riesz--Thorin theorem follows from 
\cite{GLY} Lemma~2.6.
Let $q \in (\frac{2d}{d+2},2)$.
By the extension of \cite{Aus2} Proposition~3.2(2) we obtain that $T$ satisfies 
$L_q$--$L_2$ off-diagonal estimates, 
which implies by the extension of \cite{Aus2} Proposition~3.2(3) 
that $T$ is uniformly bounded in $L_q$.
Dualizing again, we obtain the statement for all $p \in (2,\frac{2d}{d-2})$.
By considering the adjoint form, applying the result for $p>2$ and 
taking the dual we obtain the statement for $p\in(\frac{2d}{d+2},2)$.
\end{proof}

\begin{remark}
We comment on the admissible ranges for $p$ in Theorems~\ref{tpacc301} 
and~\ref{thm:vv-auscher}.
Remarkably, it is possible that the range given in Theorem~\ref{tpacc301} 
for contractive extensions is larger than the one given in 
Theorem~\ref{thm:vv-auscher} for extensions with growth bound $0$.
For example, this occurs if $\frac{\mu}{M} \ge \frac{1}{2}$, say, 
and $d$ is sufficiently large.
\end{remark}

\appendix

\section{The derivative of a truncation} \label{Spaccapp1}

Let $\Omega \subset \Ri^d$ be an open set.
Let $H$ be a Hilbert space.
The principal aim in this section is to prove the following 
chain rule.

\begin{prop} \label{ppaccapp101}
Let $\alpha > 0$ and $M > 0$.
Let $u \in H^1(\Omega,H)$.
Define $v = (\|u\|_H^\alpha \wedge M) \, u$.
Then $v \in H^1(\Omega,H)$ and 
\[
\partial_k v
= \alpha \, \one_{[\|u\|_H^\alpha < M]} \, \|u\|_H^\alpha \, \RRe(\sgn u, \partial_k u) \, \sgn u
   + (\|u\|_H \wedge M) \, \partial_k u
\]
for all $k \in \{ 1,\ldots,d \} $.
\end{prop}

The proof involves some work.
We use the following approximation by smooth functions.

\begin{lemma} \label{lpaccapp102}
The space $C^\infty(\Omega,H) \cap H^1(\Omega,H)$ is dense in $H^1(\Omega,H)$.
\end{lemma}
\begin{proof}
This follows as in the scalar case in \cite{AF} Theorem~3.17.
\end{proof}

For an approximation argument the next lemma is useful.

\begin{lemma} \label{lpaccapp103.4}
For all $n \in \Ni$ let $u_n \in H^1(\Omega,H)$.
Let $u,g_1,\ldots,g_d \in L_2(\Omega,H)$.
Suppose that $\lim u_n = u$ in $L_2(\Omega,H)$ 
and $\lim \partial_k u_n = g_k$ in $L_2(\Omega,H)$  for all $k \in \{ 1,\ldots,d \} $.
Then $u \in H^1(\Omega,H)$ and 
$\partial_k u = g_k$ for all $k \in \{ 1,\ldots,d \} $.
\end{lemma}
\begin{proof}
Let $\varphi \in C_c^\infty(\Omega)$.
Let $k \in \{ 1,\ldots,d \} $.
Then $- \int_\Omega u_n \, \partial_k \varphi = \int_\Omega (\partial_k u_n) \, \varphi$
for all $n \in \Ni$.
Then the lemma follows by taking the limit $n \to \infty$.
\end{proof}

For the proof of Proposition~\ref{ppaccapp101} we shall approximate
the function $t \mapsto t^\alpha \wedge M$ with smooth functions.
The next technical lemma gives sufficient conditions in order to 
apply a chain rule. 
Note that we do not require that $f'$ is bounded.

\begin{lemma} \label{lpaccapp104}
Let $f \in C^1(0,\infty)$.
Suppose that $f$ is bounded, $\lim_{t \downarrow 0} f(t) = 0$, 
$\lim_{t \downarrow 0} t \, f'(t) = 0$ and 
$\sup_{t \in (0,\infty)} t \, |f'(t)| < \infty$.
Let $u \in H^1(\Omega,H)$.
Define $v = f(\|u\|_H) \, u$.
Then $v \in H^1(\Omega,H)$ and 
\begin{equation}
\partial_k v
= \left\{ \begin{array}{ll}
   \|u\|_H \, f'(\|u\|_H) \, \RRe(\sgn u, \partial_k u)_H \, \sgn u
     + f(\|u\|_H) \, \partial_k u & \mbox{on } [u \neq 0],  \\[5pt]
   0 & \mbox{on } [u = 0].
          \end{array} \right.
\label{elpaccapp104;2}
\end{equation}
for all $k \in \{ 1,\ldots,d \} $.
\end{lemma}
\begin{proof}
Let $\varepsilon > 0$.
Define $v_\varepsilon = f(\sqrt{\|u\|_H^2 + \varepsilon}) \, u$.
If $u \in C^1(\Omega,H)$, then $v_\varepsilon \in C^1(\Omega,H)$ and
\begin{equation}
\partial_k v_\varepsilon
= \sqrt{\|u\|_H^2 + \varepsilon} \, f'(\sqrt{\|u\|_H^2 + \varepsilon}) \, 
   \frac{\RRe(u, \partial_k u)_H}{\sqrt{\|u\|_H^2 + \varepsilon}} \, 
  \frac{1}{\sqrt{\|u\|_H^2 + \varepsilon}} \, u
     + f(\sqrt{\|u\|_H^2 + \varepsilon}) \, \partial_k u
\label{elpaccapp104;1}
\end{equation}
for all $k \in \{ 1,\ldots,d \} $.
Then by Lemmas~\ref{lpaccapp102} and \ref{lpaccapp103.4}
this extends to all $u \in H^1(\Omega,H)$ and (\ref{elpaccapp104;1}) is valid.
Finally choose $\varepsilon = \frac{1}{n}$, take the limit $n \to \infty$
and use again Lemma~\ref{lpaccapp103.4}.
\end{proof}

Now we are able to prove Proposition~\ref{ppaccapp101}.

\begin{proof}[{\bf Proof of Proposition~\ref{ppaccapp101}.}]
For all $n \in \Ni$ define $f_n,f \colon (0,\infty) \to \Ri$ by
\begin{eqnarray*}
f(t) & = & t^\alpha \wedge M,  \\[5pt]
f_n(t) & = & \tfrac{1}{2} \, ( t^\alpha + \sqrt{M^2 + n^{-1}} - \sqrt{|t^\alpha - M|^2 + n^{-1}} ) .
\end{eqnarray*}
Then $\lim f_n(t) = f(t)$ for all $t \in (0,\infty)$.
Also $\lim_{t \downarrow 0} f_n(t) = 0$ for all $n \in \Ni$.
Let $n \in \Ni$.
Then $f_n \in C^1(0,\infty)$ and
\[
f_n'(t) 
= \tfrac{1}{2} \, \alpha \, t^{\alpha - 1} 
   \Big( 1 - \frac{t^\alpha - M}{\sqrt{(t^\alpha - M)^2 + n^{-1}}} \Big)
\]
for all $t \in (0,\infty)$.
In particular, $f_n$ is increasing.
Moreover, $\lim_{t \downarrow 0} t \, f_n'(t) = 0$.
In addition, $\lim_{n \to \infty} f_n'(t) = \alpha \, t^{\alpha - 1}$
if $t^\alpha \leq M$ and $\lim_{n \to \infty} f_n'(t) = 0$ if $t^\alpha > M$.

Let $n \in \Ni$ and $t \in (0,\infty)$.
If $t^\alpha \leq M$, then 
\[
0 \leq 
t \, f_n'(t) 
= \tfrac{1}{2} \, \alpha \, t^\alpha
    \Big( 1 + \frac{M - t^\alpha}{\sqrt{(t^\alpha - M)^2 + n^{-1}}} \Big)
\leq \alpha \, M
.  \]
Alternatively, if $t^\alpha > M$, then 
\begin{eqnarray*}
0 
& \leq & t \, f_n'(t) 
= \tfrac{1}{2} \, \alpha \, \frac{t^\alpha - M + M}{\sqrt{(t^\alpha - M)^2 + n^{-1}}} 
   \Big( \sqrt{(t^\alpha - M)^2 + n^{-1}} - \sqrt{(t^\alpha - M)^2} \Big) \\
& \leq & \tfrac{1}{2} \, \alpha \Big( 1 + \frac{M}{\sqrt{n^{-1}}} \Big)
   \sqrt{n^{-1}}
\leq \tfrac{1}{2} \, \alpha \, (1 + M)
.
\end{eqnarray*}
So 
\begin{equation}
\sup_{n \in \Ni} \sup_{t \in (0,\infty)} t \, |f'(t)|
\leq \alpha \, (M+1)
.  
\label{eppaccapp101;1}
\end{equation}
If $n \in \Ni$ and $t \in (0,\infty)$, then 
\[
0 
\leq f_n(t)
\leq \tfrac{1}{2} \, (t^\alpha + M+1 - \sqrt{|t^\alpha - M|^2 + n^{-1}} )
\leq \tfrac{1}{2} \, (t^\alpha + M+1 - |t^\alpha - M| )
= \tfrac{1}{2} + f(t)
\leq \tfrac{1}{2} + M
.  \]
So $f_n$ is bounded and even
\begin{equation}
\sup_{n \in \Ni} \sup_{t \in (0,\infty)} |f_n(t)|
\leq \tfrac{1}{2} + M
.  
\label{eppaccapp101;2}
\end{equation}
Hence all conditions of Lemma~\ref{lpaccapp104} are satisfied for all the $f_n$.

Let $u \in H^1(\Omega,H)$.
For all $n \in \Ni$ define $v_n = f_n(u) \, u$.
Then $v_n \in H^1(\Omega,H)$ with derivatives given by (\ref{elpaccapp104;2})
and $f$ replaced by $f_n$.
The Lebesgue dominated convergence theorem and the uniform bounds 
(\ref{eppaccapp101;2}) and (\ref{eppaccapp101;1})
imply that $\lim v_n = v$ and $\lim \partial_k v_n = \partial_k v$ in 
$L_2(\Omega,H)$ for all $k \in \{ 1,\ldots,d \} $.
Then the proposition follows from Lemma~\ref{lpaccapp103.4}.
\end{proof}

Almost the same arguments show that the norm of an 
$H^1(\Omega,H)$-function is in the Sobolev space.

\begin{lemma} \label{lpaccapp103.7}
Let $u \in H^1(\Omega,H)$.
Then $\|u\|_H \in H^1(\Omega)$ and 
$\partial_k \|u\|_H = \RRe(\sgn u, \partial_k u)$ for all 
$k \in \{ 1,\ldots,d \} $.
\end{lemma}
\begin{proof}
Let $\varepsilon > 0$.
For all $u \in H^1(\Omega,H)$ define $u_\varepsilon \colon \Omega \to H$ by
$u_\varepsilon = \sqrt{\|u\|_H^2 + \varepsilon}$.
Let $\varphi \in C_c^\infty(\Omega)$ and $k \in \{ 1,\ldots,d \} $.
If $u \in C^\infty(\Omega,H) \cap H^1(\Omega,H)$, 
then $u_\varepsilon \in C^\infty(\Omega,H)$ with classical 
partial derivative
$\partial_k u_\varepsilon = \frac{\RRe(u, \partial_k u)_H}{u_\varepsilon}$.
Hence 
\begin{equation} 
- \int_\Omega u_\varepsilon \, \partial_k \varphi
= \int_\Omega \frac{\RRe(u, \partial_k u)_H}{u_\varepsilon} \, \varphi
.
\label{elpaccapp103.7;1}
\end{equation}
Using approximation and Lemma~\ref{lpaccapp102}, it follows that 
(\ref{elpaccapp103.7;1}) is valid for all $u \in H^1(\Omega,H)$.
Finally choose $\varepsilon = \frac{1}{n}$ and take the limit $n \to \infty$.
\end{proof}

\section{Strict convexity of $L_p(\Omega,H)$} \label{Spaccapp2}
As before let $(\Omega,\cb,\mu)$ be a $\sigma$-finite measure space
and $H$ a Hilbert space.
In order to make this paper more self-contained, we give a direct 
proof of the following theorem.
At the end of this section we give information on more general results.

\begin{thm} \label{tpaccapp201}
Let $p \in (1,\infty)$ and $u,v \in L_p(\Omega,H)$ with 
$\|u\|_p = \|v\|_p = 1$.
If $\|u+v\|_p = 2$, then $u = v$.
\end{thm}

For the proof of Theorem~\ref{tpaccapp201} we use three lemmas.

\begin{lemma} \label{lpaccapp202}
Let $\xi,\eta \in H$ and suppose  that $\|\xi + \eta\|_H = \|\xi\|_H + \|\eta\|_H$.
If $\eta \neq 0$, then there is a $\lambda \in [0,\infty)$ such that $\xi = \lambda \, \eta$.
\end{lemma}
\begin{proof}
The equality implies that $\RRe (\xi,\eta)_H = \|\xi\|_H \, \|\eta\|_H$.
This gives equality in the Cauchy--Schwarz inequality.
Hence there is a $\lambda \in \Ci$ such that $\xi = \lambda \, \eta$.
Using again the equality, one deduces that $|1 + \lambda| = |\lambda| + 1$
and therefore $\lambda \in [0,\infty)$.
\end{proof}

Let $p,q \in (1,\infty)$ and suppose that $\frac{1}{p} + \frac{1}{q} = 1$.

\begin{lemma} \label{lpaccapp203}
Let $a,b \in [0,\infty)$.
Then $a \, b \leq \frac{1}{p} \, a^p + \frac{1}{q} \, b^q$
and the equality holds if and only if $a^p = b^q$.
\end{lemma}
\begin{proof}
This follows from the concavity of the logarithm.
\end{proof}

\begin{lemma} \label{lpaccapp204}
Let $f \in L_p(\Omega)$ and $g \in L_q(\Omega)$ with $f \neq 0$ and $g \neq 0$.
Suppose that $\int_\Omega |f| \, |g| = \|f\|_{L_p(\Omega)} \, \|g\|_{L_q(\Omega)}$.
Then there exists a $\lambda > 0$ such that 
$|f|^p = \lambda \, |g|^q$ almost everywhere.
\end{lemma}
\begin{proof}
We may assume that $\|f\|_{L_p(\Omega)} = 1 = \|g\|_{L_q(\Omega)}$.
Then Lemma~\ref{lpaccapp203} gives
\[
1 = \int_\Omega |f| \, |g|
\leq \int_\Omega \tfrac{1}{p} \, |f|^p + \tfrac{1}{q} \, |g|^q
= \tfrac{1}{p} \, \|f\|_{L_p(\Omega)} + \tfrac{1}{q} \, \|g\|_{L_q(\Omega)}
= 1
.  \]
Hence $|f| \, |g| = \tfrac{1}{p} \, |f|^p + \frac{1}{q} \, |g|^q$ almost everywhere
and the lemma follows from Lemma~\ref{lpaccapp203}.
\end{proof}

\begin{proof}[{\bf Proof of Theorem~\ref{tpaccapp201}.}]
Using the triangle inequality on $H$, twice the H\"older inequality on $L_p(\Omega)$
and the assumption $\|u+v\|_p = \|u\|_p + \|v\|_p$ one obtains
\begin{eqnarray*}
\|u + v\|_p^p
& = & \int_\Omega \|u+v\|_H \, \|u+v\|_H^{p-1}  \\
& \leq & \int_\Omega \|u\|_H \, \|u+v\|_H^{p-1} + \int_\Omega \|v\|_H \, \|u+v\|_H^{p-1}   \\
& \leq & \Big( \int_\Omega \|u\|_H^p \Big)^{1/p} \Big( \int_\Omega \|u+v\|_H^{(p-1) q} \Big)^{1/q} 
   + \Big( \int_\Omega \|v\|_H^p \Big)^{1/p} \Big( \int_\Omega \|u+v\|_H^{(p-1) q} \Big)^{1/q}   \\
& = & (\|u\|_p + \|v\|_p) \, \|u+v\|_p^{p/q} 
= \|u+v\|_p^{\frac{p}{q} + 1}  
= \|u+v\|_p^p
.
\end{eqnarray*}
Hence all three inequalities are equalities.
The first gives that there is a null-set $N_1 \subset \Omega$ such that 
$\|u+v\|_H(x) = \|u\|_H(x) + \|v\|_H(x)$ for all
$x \in \Omega \setminus N_1$ such that $\|u+v\|_H(x) \neq 0$.
Recall that $\|u\|_p = 1$, so $\|u\|_H \neq 0 \in L_p(\Omega)$.
Similarly $\|u+v\|_H \neq 0 \in L_p(\Omega)$ and therefore 
$\|u+v\|_H^{p-1} \neq 0 \in L_q(\Omega)$.
Hence the equality in the first H\"older inequality together with 
Lemma~\ref{lpaccapp204} gives that there are $\alpha > 0$ 
and a null-set $N_2 \subset \Omega$ such that 
$\|u\|_H^p(x) = \alpha \, \|u+v\|_H^{(p-1)q}(x)$ for all $x \in \Omega \setminus N_2$.
Similarly there are $\beta > 0$ and a null-set $N_3 \subset \Omega$ such that 
$\|v\|_H^p(x) = \beta \, \|u+v\|_H^{(p-1)q}(x)$ 
for all $x \in \Omega \setminus N_3$.
Hence $\|u\|_H^p = \gamma \, \|v\|_H^p$ on $\Omega \setminus (N_2 \cup N_3)$, where 
$\gamma = \frac{\alpha}{\beta}$.
Since $\|u\|_p = \|v\|_p = 1$, one deduces that $\gamma = 1$.

Now let $x \in \Omega \setminus (N_1 \cup N_2 \cup N_3)$.
If $\|u+v\|_H(x) = 0$, then 
$\|u\|_H^p(x) = \alpha \, \|u+v\|_H^{(p-1)q}(x) = 0$ and $u(x) = 0$.
Similarly $v(x) = 0$ and therefore $u(x) = v(x)$.
Alternatively, if $\|u+v\|_H(x) \neq 0$, then $v(x) \neq 0$
since $x \not\in N_3$.
Moreover, $\|u(x) + v(x)\|_H = \|u(x)\|_H + \|v(x)\|_H$
and Lemma~\ref{lpaccapp202} implies that there is a $\lambda \in [0,\infty)$
such that $u(x) = \lambda \, v(x)$.
But $\|u(x)\|_H^p = \|v(x)\|_H^p$ and consequently $u(x) = v(x)$.
Therefore $u = v$ almost everywhere.
\end{proof}

With a small modification one can prove that $L_p(\Omega,E)$ is strictly convex
if $E$ is strictly convex and $p \in (1,\infty)$.
In fact, a stronger result than Theorem~\ref{tpaccapp201} is known.
The space $L_p(\Omega,E)$ is uniformly convex
if $E$ is uniformly convex and $p \in (1,\infty)$.
See \cite{DHM} and the references therein.

\subsection*{Acknowledgement.} 
The second-named author is most grateful for the hospitality extended
to him during a fruitful stay at Ulm University.
He wishes to thank Ulm University for financial support.

\end{document}